\documentclass[reqno]{amsart}

\newtheorem{thm}{Theorem}[section]
\newtheorem{cor}[thm]{Corollary}

\newtheorem{prop}[thm]{Proposition}
\theoremstyle{definition}
\newtheorem{defn}[thm]{Definition}
\theoremstyle{remark}
\newtheorem{rem}[thm]{Remark}

\numberwithin{equation}{section}

\begin{document}
\title[Riemannian foliations
on quaternion CR-submanifolds]{\bf Riemannian foliations on
quaternion CR-submanifolds of an almost quaternion K\"{a}hler
product manifold}
\author[G.E. V\^\i lcu]{Gabriel Eduard V\^\i lcu$^{1,2}$}
\maketitle \thispagestyle{empty}

\begin{center}
\emph{Dedicated to Professor Stere Ianu\c{s} on the occasion of his
70th birthday.}
\end{center}

\begin{abstract} The purpose of this paper is to study the canonical
foliations of a quaternion CR-submanifold of an
almost quaternion K\"{a}hler product manifold.
\end{abstract}

\section{Introduction}

The notion of CR-submanifold, first introduced in K\"{a}hler
geometry by Bejancu (see \cite{BJC}), was extended in the quaternion
settings by Barros, Chen and Urbano in \cite{BCU}. They consider
CR-quaternion submanifolds of quaternion K\"{a}hlerian manifolds as
generalizations of both quaternion and totally real submanifolds.
Some foliations on this kind of submanifolds have been studied in
\cite{IIVL} and \cite{PS}.

The natural product of two K\"{a}hlerian manifolds is also a
K\"{a}hlerian manifold (\cite{YK}) and the geometry of
CR-submanifolds of K\"{a}hlerian product manifolds is an interesting
subject which was studied in \cite{ATC} and \cite{SHD}. On the other
hand, the product of two quaternion K\"{a}hler manifolds does not
become a quaternion K\"{a}hler manifold, but it is an almost
quaternion K\"{a}hler product manifold (see \cite{KN}).

The study of quaternion CR-submanifolds of an almost quaternion
K\"{a}hler product manifold was initiated in \cite{KL} by Kang and
Lee. If $M$ is a such kind of submanifold then two distributions,
denoted by $D$ and $D^\perp$, are defined on $M$. Moreover, if
$D^\perp$ is invariant   with  respect to the canonical almost
product structure $F$ on $M$, it follows that $D^\perp$ is always
integrable and so we have a foliation on $M$, called canonical
totally real foliation on $M$; on the other hand, if $M$ is a
$D$-geodesic CR-submanifold and $D$ is $F$-invariant then $D$ is
also integrable and so we have another foliation on $M$, called
quaternion foliation. The object of this note is to study these
foliations.

\section{Preliminaries}

Let $\overline{M}$ be a smooth manifold endowed with a tensor $F$ of
type (1,1) such that $F^{2}=Id.$ Then $(\overline{M},F)$ is said to
be an almost product manifold with almost product structure $F$.
Moreover if $\overline{g}$ is a Riemannian metric on $M$ such that:
$$\overline{g}(FX,FY)=\overline{g}(X,Y),\ \forall X,Y\in\Gamma(T\overline{M}),$$
then we say that $(\overline{M},F,\overline{g})$ is an almost
product Riemannian manifold.

\begin{defn} (\cite{KN})
Let $(\overline{M},F,\overline{g})$ be an almost product Riemannian
manifold of dimension $4n$ and assume that there is a rank
3-subbundle $\sigma$ of $End(T\overline{M})$ satisfying the
following conditions: \\
i. A local basis
$\lbrace{J_1,J_2,J_3}\rbrace$ exists of sections of $\sigma$ such
that:
       \begin{equation}\label{1}
       J_1^2=J_2^2=J_3^2=-Id,\ J_1J_2=-J_2J_1=J_3
       \end{equation}
and
         \begin{equation}\label{2}
         \overline{g}(J_\alpha X,J_\alpha Y)=
         \overline{g}(X,Y),\ \alpha\in\{1,2,3\}
         \end{equation}
for all local vector fields $X$, $Y$ on
$\overline{M}$.\\
ii. The Levi-Civita connection $\overline{\nabla}$ of $\overline{g}$
satisfies:
 \begin{eqnarray}\label{3}
                 (\overline{\nabla}_XJ_1)Y=k[\ \ \omega_3(X)J_2Y-\omega_2(X)J_3Y+\omega_3(FX)J_2(FY)-\omega_2(FX)J_3(FY)] \nonumber \\
                 (\overline{\nabla}_XJ_2)Y=k[-\omega_3(X)J_1Y+\omega_1(X)J_3Y-\omega_3(FX)J_1(FY)+\omega_1(FX)J_3(FY)] \nonumber \\
                 (\overline{\nabla}_XJ_3)Y=k[\ \
                 \omega_2(X)J_1Y-\omega_1(X)J_2Y+\omega_2(FX)J_1(FY)-\omega_1(FX)J_2(FY)]
                 \nonumber\\
    \end{eqnarray}
for some non-zero constant $k$ and all vector field $X$, $Y$ on
$\overline{M}$, $\omega_1,\omega_2,\omega_3$ being local 1-forms
over the open for which $\lbrace{J_1,J_2,J_3}\rbrace$ is a local
basis of $\sigma$.

Then $(\overline{M},F,\sigma,\overline{g})$ is said to be an almost
quaternion K\"{a}hler product manifold.
\end{defn}

\begin{rem}
The natural product manifold of two quaternion K\"{a}hler manifolds
is an almost quaternion K\"{a}hler product manifold. In this case we
have $k=\frac{1}{2}$ (see \cite{KN}).
\end{rem}

\begin{defn} (\cite{BCU}) A submanifold $M$ of an almost quaternion K\"{a}hler
product manifold $(\overline{M},F,\sigma,\overline{g})$ is called a
quaternion CR-submanifold if there exist two orthogonal
complementary
distributions $D$ and $D^\perp$ on $M$ such that: \\
i. $D$ is invariant under quaternion structure, that is:
\begin{equation}\label{4}
J_\alpha(D_x)\subseteq D_x, \forall x\in M,\
\forall\alpha\in\{1,2,3\};
\end{equation}
ii. $D^\perp$ is totally real, that is:
\begin{equation}\label{5}
J_\alpha(D_x^\perp)\subseteq T_xM^\perp, \forall\alpha\in\{1,2,3\},\
\forall x\in M.
\end{equation}
\end{defn}

A submanifold $M$ of an almost quaternion K\"{a}hler product
manifold $(\overline{M},F,\sigma,\overline{g})$ is a quaternion
submanifold (respectively, a totally real submanifold) if dim
$D^\perp=0$ (respectively, dim $D=0$).

A submanifold $M$ of an almost quaternion K\"{a}hler product
manifold $(\overline{M},F,\sigma,\overline{g})$ is called
$F$-invariant if $F(T_xM)\subset T_xM,\ \forall x\in M$.

The distribution $D$ (respectively $D^\perp$) is said to be
$F$-invariant if $F(D)\subset D$ (respectively $F(D^\perp)\subset
D^\perp$).

\begin{defn} (\cite{BCU}) Let $M$ be a quaternion CR-submanifold of
an almost quaternion K\"{a}hler product manifold
$(\overline{M},F,\sigma,\overline{g})$. Then $M$ is called a
QR-product if $M$ is locally the Riemannian product of a quaternion
submanifold and a totally real submanifold of $\overline{M}$.
\end{defn}

\begin{rem}
For a submanifold $M$ of an almost quaternion K\"{a}hler product
manifold $(\overline{M},F,\sigma,\overline{g})$, we denote by $g$
the metric tensor induced on $M$. If $\nabla$ is the covariant
differentiation induced on $M$, the Gauss and Weingarten formulas
are given by:
\begin{equation}\label{6}
       \overline{\nabla}_XY=\nabla_XY+B(X,Y), \forall X,Y \in
\Gamma(TM)
       \end{equation}
and
\begin{equation}\label{7}
       \overline{\nabla}_XN=-A_NX+\nabla_{X}^{\perp}N, \forall X\in
\Gamma(TM), \forall N\in \Gamma(TM^\perp)
       \end{equation}
where $B$ is the second fundamental form of $M$, $\nabla^\perp$ is
the connection on the normal bundle and $A_N$ is the shape operator
of $M$ with respect to $N$. The shape operator $A_N$ is related to
$B$ by:
\begin{equation}\label{8}
       g(A_NX,Y)=\overline{g}(B(X,Y),N),
       \end{equation}
for all $ X,Y\in \Gamma(TM)$ and $N\in \Gamma(TM^\perp)$.
\end{rem}

\begin{defn}\label{2.7} (\cite{BCU})
Let $M$ be a quaternion CR-submanifold of an almost quaternion
K\"{a}hler product manifold $(\overline{M},\sigma,\overline{g})$.
The we say that:
\\
i. $M$ is $D$-geodesic if $B(X,Y)=0, \forall X,Y\in\Gamma(D).$\\
ii. $M$ is $D^\perp$-geodesic if $B(X,Y)=0, \forall
X,Y\in\Gamma(D^\perp).$\\
iii. $M$ is mixed geodesic if $B(X,Y)=0, \forall X\in\Gamma(D),
Y\in\Gamma(D^\perp).$
\end{defn}

We recall now the following results which we shall need in the
sequel.

\begin{thm} $($\cite{KL}$)$\label{2.8}
If $M$ is a quaternion CR-submanifold of an almost quaternion
K\"{a}hler product manifold $(\overline{M},F,\sigma,\overline{g})$
such that the totally real distribution $D^\perp$ is $F$-invariant,
then $D^\perp$ is integrable.
\end{thm}

\begin{thm} $($\cite{KL}$)$\label{2.9}
If $M$ is a quaternion CR-submanifold of an almost quaternion
K\"{a}hler product manifold $(\overline{M},F,\sigma,\overline{g})$
such that the quaternion distribution $D$ is $F$-invariant, then $D$
is integrable if and only if $M$ is $D$-geodesic.
\end{thm}

\section{Totally real foliation on a quaternion CR-submanifold}

Let $M$ be a quaternion CR-submanifold of an almost quaternion
K\"{a}hler product manifold $(\overline{M},F,\sigma,\overline{g})$.
Then we have the orthogonal decomposition:
$$TM=D\oplus D^\perp.$$

We have also the following orthogonal decomposition:
$$TM^\perp=\mu\oplus\mu^\perp,$$
where $\mu$ is the subbundle of the normal bundle $TM^\perp$ which
is the orthogonal complement of:
$$\mu^\perp=J_1D^\perp\oplus J_2D^\perp\oplus J_3D^\perp.$$

If the totally real distribution $D^\perp$ is $F$-invariant, by
Theorem \ref{2.8} we can consider the foliation $\mathfrak{F}^\perp$
on $M$, with structural distribution $D^\perp$ and transversal
distribution $D$, called the canonical totally real foliation on
$M$.

We can illustrate now some of the techniques in this paper on the
following theorem (see also \cite{BJC2}, \cite{IIV}, \cite{IIVL},
\cite{KL}).

\begin{thm}\label{3.1}
If $M$ is a quaternion CR-submanifold of an almost quaternion
K\"{a}hler product manifold $(\overline{M},F,\sigma,\overline{g})$
such that the totally real distribution $D^\perp$ is $F$-invariant,
then the next assertions are equivalent:
\\
i. The canonical totally real foliation $\mathfrak{F}^\perp$ is
totally geodesic;
\\
ii. $B(X,Y)\in\Gamma(\mu)$, $\forall X\in\Gamma(D)$,
$Y\in\Gamma(D^\perp)$;
\\
iii. $A_NX\in\Gamma(D^\perp)$, $\forall X\in\Gamma(D^\perp)$,
$N\in\Gamma(\mu^\perp)$;
\\
iv. $A_NY\in\Gamma(D)$, $\forall Y\in\Gamma(D)$,
$N\in\Gamma(\mu^\perp)$.
\end{thm}
\begin{proof}
i. $\Leftrightarrow$ ii. For $X,Z\in\Gamma(D^\perp)$ and
$Y\in\Gamma(D)$ we have:
\begin{eqnarray}
\overline{g}(J_\alpha(\nabla_XZ),Y)&=&-\overline{g}(\overline{\nabla}_XZ-B(X,Z),J_\alpha Y)\nonumber\\
&=&\overline{g}(-(\overline{\nabla}_XJ_\alpha)Z+\overline{\nabla}_XJ_\alpha
       Z,Y)\nonumber\\
       &=&k\overline{g}(\omega_\beta(X)J_\gamma
Z-\omega_\gamma(X)J_\beta
Z,Y)\nonumber\\&&+k\overline{g}(\omega_\beta(FX)J_\gamma(FZ)-\omega_\gamma(FX)J_\beta(FZ),Y)\nonumber\\
       &&+\overline{g}(-A_{J_\alpha Z}X+\nabla_X^{\perp}J_\alpha
Z,Y)\nonumber\\
       &=&-g(A_{J_\alpha Z}X,Y)\nonumber
       \end{eqnarray}
where $(\alpha,\beta,\gamma)$ is an even permutation of (1,2,3), and
taking into account (\ref{7}) we obtain:
\begin{equation}\label{13bis}
\overline{g}(J_\alpha(\nabla_XZ),Y)=-\overline{g}(B(X,Y),J_\alpha
Z).
\end{equation}

The equivalence is now clear from (\ref{13bis}).
\\
ii. $\Leftrightarrow$ iii. This equivalence follows from (\ref{8}).
\\
iii. $\Leftrightarrow$ iv. This equivalence is true because $A_N$ is
a self-adjoint operator.
\end{proof}

\begin{prop}\label{3.3}
If $M$ is a quaternion CR-submanifold of an almost quaternion
K\"{a}hler product manifold $(\overline{M},F,\sigma,\overline{g})$
such that the totally real distribution $D^\perp$ is $F$-invariant
and $M$ is mixed geodesic, then the canonical totally real foliation
$\mathfrak{F}^\perp$ is totally geodesic.
\end{prop}
\begin{proof}
The assertion follows from Theorem \ref{3.1}.
\end{proof}

A submanifold $M$ of a Riemannian manifold
$(\overline{M},\overline{g})$ is said to be a ruled submanifold if
it admits a foliation whose leaves are totally geodesic immersed in
$(\overline{M},\overline{g})$.

\begin{defn}
A quaternion CR-submanifold of an almost quaternion K\"{a}hler
product manifold which is a ruled submanifold with respect to the
foliation $\mathfrak{F}^\perp$ is called totally real ruled
quaternion CR-submanifold.
\end{defn}

\begin{thm}\label{4.2}
Let $M$ be a quaternion CR-submanifold of an almost quaternion
K\"{a}hler product manifold $(\overline{M},F,\sigma,\overline{g})$
such that $D^\perp$ is $F$-invariant. The next assertions are
equivalent:
\\
i. $M$ is a totally real ruled quaternion CR-submanifold.
\\
ii. $M$ is $D^\perp$-geodesic and:
$$B(X,Y)\in\Gamma(\mu),\ \forall
X\in\Gamma(D),\ Y\in\Gamma(D^\perp).$$ iii. The subbundle
$\mu^\perp$ is $D^\perp$-parallel, i.e:
$$\nabla_X^\perp J_\alpha Z\in\Gamma(\mu^\perp),\ \forall X,Z\in
D^\perp,\ \alpha\in\{1,2,3\}$$ and the second fundamental form
satisfies:
$$B(X,Y)\in\Gamma(\mu),\ \forall
X\in\Gamma(D^\perp),\ Y\in\Gamma(TM).$$ iv. The shape operator
satisfies:
$$A_{J_\alpha Z}X=0,\ \forall X,Z\in D^\perp,\ \alpha\in\{1,2,3\}$$
and
$$A_N X\in\Gamma(D),\ \forall X\in\Gamma(D^\perp),\ N\in\Gamma(\mu).$$
\end{thm}
\begin{proof}
i. $\Leftrightarrow$ ii. For any $X,Z\in\Gamma(D^\perp)$ we have:
\begin{eqnarray}
       \overline{\nabla}_XZ&=&\nabla_XZ+B(X,Z)\nonumber\\
       &=&\nabla_X^{D^\perp}Z+h^\perp(X,Z)+B(X,Z)\nonumber
       \end{eqnarray}
and thus we conclude that the leaves of $D^\perp$ are totally
geodesic immersed in $\overline{M}$ if and only if  $h^\perp=0$ and
$M$ is $D^\perp$-geodesic. The equivalence is now clear from Theorem
\ref{3.1}.
\\
i. $\Leftrightarrow$ iii. For any $X,Z\in\Gamma(D^\perp)$, and
$U\in\Gamma(D)$ we have:
\begin{eqnarray}
\overline{g}(\overline{\nabla}_XZ,U)&=&\overline{g}(J_\alpha\overline{\nabla}_XZ,J_\alpha U)\nonumber\\
&=&\overline{g}(-(\overline{\nabla}_XJ_\alpha)Z+\overline{\nabla}_XJ_\alpha
Z,J_\alpha
       U)\nonumber\\
       &=&k\overline{g}(\omega_\beta(X)J_\gamma
Z-\omega_\gamma(X)J_\beta
Z,J_\alpha U)\nonumber\\&&+k\overline{g}(\omega_\beta(FX)J_\gamma(FZ)-\omega_\gamma(FX)J_\beta(FZ),J_\alpha U)\nonumber\\
       &&+\overline{g}(-A_{J_\alpha Z}X+\nabla_X^{\perp}J_\alpha
Z,J_\alpha U)\nonumber\\
       &=&-g(A_{J_\alpha Z}X,J_\alpha U)\nonumber
       \end{eqnarray}
where $(\alpha,\beta,\gamma)$ is an even permutation of (1,2,3), and
from (\ref{8}) we obtain:
\begin{equation}\label{10}
\overline{g}(\overline{\nabla}_XZ,U)=-\overline{g}(B(X,J_\alpha
U),J_\alpha Z).
\end{equation}

On the other hand, for any $X,Z,W\in\Gamma(D^\perp)$ we have:
\begin{eqnarray}\label{11}
       \overline{g}(\overline{\nabla}_XZ,J_\alpha
W)=\overline{g}(B(X,Z),J_\alpha W).
       \end{eqnarray}

If $X,Z\in\Gamma(D^\perp)$ and $N\in\Gamma(\mu)$, then we have:
\begin{eqnarray}
\overline{g}(\overline{\nabla}_XZ,N)&=&\overline{g}(J_\alpha\overline{\nabla}_XZ,J_\alpha N)\nonumber\\
&=&\overline{g}(-(\overline{\nabla}_XJ_\alpha)Z+\overline{\nabla}_XJ_\alpha Z,J_\alpha N)\nonumber\\
       &=&k\overline{g}(\omega_\beta(X)J_\beta
Z+\omega_\gamma(X)J_\gamma
       Z,N)\nonumber\\&&+k\overline{g}(\omega_\beta(FX)J_\beta(FZ)+\omega_\gamma(FX)J_\gamma(FZ),N)\nonumber\\
       &&+\overline{g}(-A_{J_\alpha Z}X+\nabla_X^{\perp}J_\alpha
Z,J_\alpha N)\nonumber
       \end{eqnarray}
and thus we obtain:
\begin{equation}\label{12}
\overline{g}(\overline{\nabla}_XZ,N)=\overline{g}(\nabla^\perp_X
J_\alpha Z,J_\alpha N).
\end{equation}

Finally, $M$ is a totally real ruled quaternion CR-submanifold if
and only if $\overline{\nabla}_X Z\in\Gamma(D^\perp)$, $\forall
X,Z\in\Gamma(D^\perp)$ and by using (\ref{10}), (\ref{11}) and
(\ref{12}) we deduce the equivalence.
\\
ii. $\Leftrightarrow$ iv. This is clear from (\ref{8}).
\end{proof}

\begin{cor}\label{4.3}
Let $M$ be a quaternion CR-submanifold of an almost quaternion
K\"{a}hler product manifold $(\overline{M},F,\sigma,\overline{g})$
such that $D^\perp$ is $F$-invariant. If $M$ is totally geodesic,
then $M$ is a totally real ruled quaternion CR-submanifold.
\end{cor}
\begin{proof}
The assertion is clear from Theorem \ref{4.2}.
\end{proof}

\section{QR-products and foliations with bundle-like metric}

From Theorem \ref{2.9} we deduce that any $D$-geodesic
CR-submanifold of an almost quaternion K\"{a}hler product manifold
such that $D$ is $F$-invariant, admits a $\sigma$-invariant totally
geodesic foliation, which we denote by $\mathfrak{F}$, called
quaternion foliation.

\begin{prop}\label{6.2}
If $M$ is a totally geodesic quaternion CR-submanifold of an almost
quaternion K\"{a}hler product manifold
$(\overline{M},F,\sigma,\overline{g})$ such that $D$ and $D^\perp$
are $F$-invariant, then $M$ is a ruled submanifold with respect to
both foliations $\mathfrak{F}$ and $\mathfrak{F^\perp}$.
\end{prop}
\begin{proof}
The assertion follows from Theorem \ref{2.9} and Corollary
\ref{4.3}.
\end{proof}

\begin{thm}\label{6.3}
Let $M$ be a quaternion CR-submanifold of an almost quaternion
K\"{a}hler product manifold $(\overline{M},F,\sigma,\overline{g})$
such that $D$ and $D^\perp$ are $F$-invariant. Then $M$ is a
QR-product if and only if the next two conditions are satisfied:
\\
i. $M$ is $D$-geodesic.
\\
ii. $B(X,Y)\in\Gamma(\mu),\ \forall X\in\Gamma(D^\perp),\
Y\in\Gamma(D)$.
\end{thm}
\begin{proof}
The proof is immediate from Theorems \ref{2.9} and \ref{3.1}.\\
\end{proof}

Let $(M,g)$ be a Riemannian manifold and $\mathfrak{F}$ a foliation
on $M$. The metric $g$ is said to be bundle-like for the foliation
$\mathfrak{F}$ if the induced metric on the transversal distribution
$D^\perp$ is parallel with respect to the intrinsic connection on
$D^\perp$. This is true if and only if the Levi-Civita connection
$\nabla$ of $(M,g)$ satisfies (see \cite{BJC2}):
\begin{equation}\label{13}
g(\nabla_{Q^\perp Y}QX,Q^\perp Z)+g(\nabla_{Q^\perp Z}QX,Q^\perp
Y)=0,\ \forall X,Y,Z\in\Gamma(TM),
\end{equation}
where $Q^\perp$ is the projection morphism of $TM$ on $D^\perp$.

If for a given foliation $\mathfrak{F}$ there exists a Riemannian
metric $g$ on $M$ which is bundle-like for $\mathfrak{F}$, then we
say that $\mathfrak{F}$ is a Riemannian foliation on $(M,g)$.

\begin{thm}\label{5.1}
Let $M$ be a quaternion CR-submanifold of an almost quaternion
K\"{a}hler product manifold $(\overline{M},F,\sigma,\overline{g})$
such that $D^\perp$ is $F$-invariant. The next assertions are
equivalent:
\\
i. The induced metric $g$ on $M$ is bundle-like for the totally real
foliation $\mathfrak{F}^\perp$.
\\
ii. The second fundamental form $B$ of $M$ satisfies:
$$B(U,J_\alpha
V)+B(V,J_\alpha U)\in\Gamma(\mu)\oplus J_\beta(D^\perp)\oplus
J_\gamma(D^\perp),\ \forall U,V\in\Gamma(D)$$ for $\alpha=1,2$ or
$3$, where $(\alpha,\beta,\gamma)$ is an even permutation of
$(1,2,3)$.
\end{thm}
\begin{proof}
From (\ref{13}) we deduce that $g$ is bundle-like for totally real
foliation $\mathfrak{F}^\perp$ if and only if:
\begin{equation}\label{14}
g(\nabla_U X,V)+g(\nabla_V X,U)=0,\ \forall X\in\Gamma(D^\perp),\
U,V\in\Gamma(D).
\end{equation}

On the other hand, for any $X\in\Gamma(D^\perp),\ U,V\in\Gamma(D)$
we have:
\begin{eqnarray}
       g(\nabla_U X,V)+g(\nabla_V
X,U)&=&\overline{g}(\overline{\nabla}_U X-B(U,X),V)+
       \overline{g}(\overline{\nabla}_V X-B(V,X),U)\nonumber\\
       &=&\overline{g}(\overline{\nabla}_U X,V)+
       \overline{g}(\overline{\nabla}_V X,U)\nonumber\\
       &=&\overline{g}(-(\overline{\nabla}_U
J_\alpha)X+\overline{\nabla}_U J_\alpha X,J_\alpha
V)\nonumber\\&&+\overline{g}(-(\overline{\nabla}_V
J_\alpha)X+\overline{\nabla}_V J_\alpha X,J_\alpha U)\nonumber\\
       &=&k\overline{g}(\omega_\beta(U)J_\gamma
X-\omega_\gamma(U)J_\beta
X,J_\alpha V)\nonumber\\&&+k\overline{g}(\omega_\beta(FU)J_\gamma(FX)-\omega_\gamma(FU)J_\beta(FX),J_\alpha V)\nonumber\\
       &&+k\overline{g}(\omega_\beta(V)J_\gamma
X-\omega_\gamma(V)J_\beta
X,J_\alpha U)\nonumber\\&&+k\overline{g}(\omega_\beta(FV)J_\gamma(FX)-\omega_\gamma(FV)J_\beta(FX),J_\alpha U)\nonumber\\
       &&+\overline{g}(\overline{\nabla}_U J_\alpha X,J_\alpha V)
       +\overline{g}(\overline{\nabla}_V J_\alpha X,J_\alpha
U)\nonumber\\
       &=&-g(A_{J_\alpha X} U,J_\alpha V)-g(A_{J_\alpha X}V, J_\alpha
U)\nonumber
       \end{eqnarray}
where $(\alpha,\beta,\gamma)$ is an even permutation of (1,2,3), and
taking into account (\ref{8}) we derive:
\begin{equation}\label{15}
g(\nabla_U X,V)+g(\nabla_V X,U)=-\overline{g}(B(U,J_\alpha
V)+B(V,J_\alpha U),J_\alpha X),
\end{equation}
for any $X\in\Gamma(D^\perp),\ U,V\in\Gamma(D)$.

The proof is now complete from (\ref{14}) and (\ref{15}).
\end{proof}
\textbf{Acknowledgement}\\ The author expresses his gratitude to the
referee for carefully reading the manuscript and giving useful
comments. This work was partially supported by a PN2-IDEI grant, no.
525/2009.

\address{
$^1$Department of Mathematics and Computer Science,

Petroleum-Gas University of Ploie\c sti,

Bulevardul Bucure\c sti, Nr. 39, Ploie\c sti 100680, Romania

$^2$Research Center in Geometry, Topology and Algebra,

Faculty of Mathematics and Computer Science,

University of Bucharest,

Str. Academiei, Nr.14, Sector 1, Bucharest 70109, Romania\\

E-mail: gvilcu@mail.upg-ploiesti.ro}

\end{document}